\documentclass[reqno]{amsart}
\usepackage{amsmath,amsfonts,euscript,amscd,amsthm,amssymb,upref,graphics,color,verbatim}
\usepackage[normalem]{ulem}
\usepackage{enumitem}
\usepackage{graphics}
\usepackage{epsfig}

\theoremstyle{plain}
\newtheorem{theorem}{Theorem}
\newtheorem{proposition}[theorem]{Proposition}
\newtheorem{lemma}[theorem]{Lemma}

\theoremstyle{definition}

\newtheorem{remark}[subsection]{Remark}

\newtheorem{nothing*}[subsection]{}

\theoremstyle{remark}
\newcommand{\rien}[1]{}

\newcommand{\C}{\ensuremath{\mathbb{C}}}
\newcommand{\B}{\ensuremath{\mathbb{B}}}
\newcommand{\N}{\ensuremath{\mathbb{N}}}

\def\e{\epsilon}

\renewcommand{\epsilon}{\varepsilon}
\renewcommand{\phi}{\varphi}
\renewcommand{\emptyset}{\varnothing}

\title{Automorphisms of $\C^m$ with bounded wandering domains}

\author{Luka Boc Thaler}

\begin{document}

\address{L. Boc Thaler: Faculty of Education, University of Ljubljana, SI--1000 Ljubljana, Slovenia.} \email{luka.boc@pef.uni-lj.si}

\begin{abstract}
 We prove that the Euclidean ball can be realized as a Fatou component of a holomorphic automorphism of  $\C^m$, in particular as the escaping and the oscillating wandering domain. Moreover, the same is true for a large class of bounded domains, namely for all bounded regular open sets $\Omega\subset \C^m$ whose closure is polynomially convex.  Our result gives in particular the first example of a bounded Fatou component with a smooth boundary in the category of holomorphic automorphisms. 
\end{abstract}

\subjclass[2010]{32H50,	32M17}
\keywords{holomorphic dynamics, automorphisms, Fatou components, wandering domains, polynomially convex, Anders\'en-Lempert theory,}

\vfuzz=2pt

\vskip 1cm

\maketitle

\section{Introduction}
With the emergence of new techniques, the study of wandering domains has flourished in recent years. Many strong results have been established in the category of transcendental entire functions \cite{Ba,Bi,EL,BEGRS,MPS} and in the category of holomorphic endomorphisms of $\C\mathbb{P}^2$ \cite{ABDPR, AsBTP} especially about the existence and the geometry of wandering domains. On the other hand there are only few known results about wandering domains in the category of holomorphic automorphisms of $\C^m$ for $m\geq2$, some of which will be presented later in the introduction. The aim of this paper is to study the geometry of wandering domains for holomorphic automorphisms. In particular we investigate which bounded domains can be realized as a wandering Fatou component of an automorphism. The following are the main results of this paper.

\begin{theorem}\label{thmescaping}For every $m\geq2$ there exists an automorphism of $\C^m$ with an escaping wandering domain equal to the Euclidean ball. 
\end{theorem}

\begin{theorem}\label{thmoscillating}For every $m\geq2$ there exists an automorphism of $\C^m$ with an oscillating wandering domain equal to the Euclidean ball. 
\end{theorem}

As we will argue in the last section, the proofs of these two theorems can easily be modified so that the same statements hold if the Euclidean ball is replaced by any bounded simply connected regular open set $\Omega\subset \C^m$ whose closure is polynomially convex, in particular by any bounded convex domain. 
\medskip 

\subsection{Background and overview of results}
Let $F$ be a holomorphic automorphism of $\C^m$. The \emph{Fatou set} $\mathcal{F} $ is the largest open subset of $\C^m$ on which the family of iterates $(F^n)_{n\geq0}$ is locally equicontinuous. A connected component $\mathcal{F}_0$ of the Fatou set is called a Fatou component and we say that such component is \emph{wandering} if and only if $F^i(\mathcal{F}_0)\cap F^j(\mathcal{F}_0)=\emptyset$ for every pair of integers $i\neq j$. We will call a wandering Fatou component a \emph{wandering domain}. There are three types of wandering domains: 
\begin{enumerate}
\item \emph{escaping}; if all orbits converge to the line at infinity,
\item \emph{oscillating}; if there exists an unbounded orbit and an orbit with a bounded subsequence,
\item \emph{orbitally bounded}; if every orbit is bounded.
\end{enumerate}

The first construction of a holomorphic automorphism of $\C^2$ with a wandering domain is due to Forn\ae ss-Sibony \cite{FS} and their wandering domain is of the oscillating type. More recently Arosio-Benini-Forn\ae ss-Peters \cite{ABFP} constructed transcendental H\'enon maps, i.e. a holomorphic automorphism of $\C^2$ of the form $F(z,w) = (f(z) + aw, az)$ with 
$f:\C \rightarrow \C$ a transcendental function, that admit wandering domains.  In particular they construct examples of wandering domains which are of the escaping and of the oscillating type and they are biholomorphic to $\C^2$. The first example of a polynomial automorphism (of $\C^4$) with a wandering domain was given by Hahn-Peters \cite{HH} and their example was of the orbitally bounded type. In \cite{ABTP} we have shown that oscillating wandering domains of transcendental H\'enon maps can also have different complex structures. In particular we have constructed a wandering domain that supports a non-constant bounded plurisubharmonic function and therefore it can not be biholomorphic to $\C^2$. The most recent result is due to Berger-Biebler \cite{BB} who  have solved a long standing problem by proving existence of polynomial H\'enon maps which admit wandering domains and note that those can only be of the orbitally bounded type.  \medskip

In this paper we prove that every bounded convex domain in $\C^m$ can be realized as an escaping or oscillating wandering domain of an automorphism of of $\C^m$. Not only does this result give new examples of wandering domains but it also provides better insight about the geometry of the Fatou components. So far the only known examples of bounded Fatou components were the Siegel balls for polynomial H\'enon maps. We know that boundary of Siegel balls is not smooth and that such a domain must be biholomorphic to one of the following three: the  polydisc, the unit ball and a Thullen domain, but which of them can be realised as the Siegel ball for polynomial H\'enon maps it is presently unknown. 

Our results give the first examples of bounded Fatou components with a smooth boundary. Note that the only other known example of an automorphism that has a Fatou component with a smooth boundary is a shear automorphism of $\C^2$, which was constructed recently by Berteloot-Cheraghi \cite{BC}. Their automorphism has an invariant Fatou component which is equal to $\C\times \Delta$ where $\Delta$ denotes the Euclidean unit disc.


We believe that the construction behind Theorem 1 will open the way for the study of  intrinsic dynamics in wandering domains, as has recently been done for transcendental entire functions in \cite{BEGRS}.

Finally let us mention that tools of Anders\'en--Lempert theory, on which our construction rely, apply to the large class of Stein manifolds with the density property, see \cite[Section 4]{For}. Therefore we strongly believe that in many of those cases, the constructions presented in this paper could easily be modified  and used to produce  automorphisms with different types of wandering domains.
\medskip

Our paper is organized as follows:

 In Section 2 we introduce the notation and recall the basic ingredients that will be used in the paper. 
 
  In Section 3 we modify the constructions from \cite{ABFP,ABTP} and reprove the existence of wandering domains biholomorphic to $\C^m$ and to Short $\C^m$ using tools of Anders\'en--Lempert theory. This modification will serve as the basis for the construction of the oscillating wandering ball in Section 5.
  
   In Section 4 we use tools of Anders\'en--Lempert theory to inductively construct a sequence of automorphisms $(F_k)$, that converge uniformly on compacts, to an automorphism $F$ with the following properties: (1) The Euclidean diameter of $F^k(\B(P_0,1))$ is less than $2$ for all $k\geq 0$, (2) $F^k(\B(P_0,1))\rightarrow \infty$ as $k\rightarrow \infty$ and (3) there is a sequence of points $(T^j_0)$ that accumulate densely on the $b\B(P_0,1)$ and each of them is contained in the basin of some attracting fixed point. This implies that the Fatou component is exactly the ball $\B(P_0,1)$ which settles  Theorem 1.  
   
In Section 5 we prove Theorem 2 by carefully combining two constructions previously introduced in Section 3 and Section 4. Moreover we argue how these constructions can be generalized to a larger class of bounded domains.

\section{Preliminaries}
In this section, we introduce the notation and recall the basic ingredients that will be used in the paper. Throughout this paper we will always assume that $m\geq2$. 

The {\em polynomially-convex hull} of a compact set $K\subset \C^m$ is defined as
\[
	\widehat{K}= \{z\in \C^m : |p(z)| \le \sup_K |p| \ \ \text{for all holomorphic polynomials $p$} \}.
\] 
We say that $K$ is \emph{polynomially convex} if $\widehat{K}=K$. Note that holomorphic automorphisms of $\C^m$ preserve polynomial convexity, and therefore any image of the closed Euclidean ball under the automorphism is polynomially convex.

Given a point $z_0\in\C^m$, we denote by $\B(z_0,r)\subset \C^m$  the open $m$-dimensional Euclidean ball of radius $r$ centered at $z_0$. We will write $\B= \B(0,1)$.

We shall frequently use the fact that the union of any two disjoint closed Euclidean balls is polynomially convex, and also the following basic result; see e.g.\cite{Stout2007}.

\begin{lemma}\label{lem:stability}
Assume that $K \subset \C^n$ is a compact polynomially convex set.
For any finite set $p_1,\ldots,p_k\in\C^m\setminus K$ and for all sufficiently
small numbers $r_1>0,\ldots,r_k>0$, the set $\bigcup_{j=1}^k \overline{\B}(p_j,r_j) \cup K$
is polynomially convex. 
\end{lemma}

A domain $D\subset \C^m$ is called \emph{starshaped} (in some literature \emph{star-like}) if there exists a point $p\in D$ such that the line segment between $p$ and any other point $q\in D$ is contained in $D$. 
 We will say that a domain $D\subset\C^m$ is  \emph{starshapelike} if there exists $\Phi$ an automorphism of $\C^m$ and a  starshaped domain $D'$ so that $D=\Phi(D')$. For example, any image of the Euclidean ball under an automorphism of $\C^m$ is a starshapelike domain whose closure is polynomially convex.

The key ingredient in our proofs will be the following result of the Anders\'en--Lempert theory, which is a corollary of \cite[Theorem 4.12.1]{For}, see also \cite[Corollary 4.12.4]{For}. 

\begin{theorem}\label{al}
Let $A_1,A_2,\ldots, A_n$ be pairwise disjoint compact sets in $\C^m$ such that all but one are
starshapelike. Let $q_j\in Aut(\C^m)$ ($j=1,\ldots,n$) be such that the images $B_j=q_j(A_j)$ are
 pairwise disjoint. If the sets $K=\cup^n_{j=1}A_j$ and $K'=\cup_{j=1}^n B_j$ are polynomially
convex, then for every $\epsilon>0$ there exists $g\in Aut(\C^m)$ such that $\|g(z)-q_j(z)\|<\epsilon$ for all $z\in A_j$, $j=1,\ldots,m$. In particular
the automorphism  $g$  can be chosen so that its finite order jets agree with the corresponding jets of $q_j$ at any given finite set of points in $A_j$, for $1\leq j\leq m$.
\end{theorem} 

Let us remark that in the above theorem the compact set $A_j$ can also be a point, since by Lemma \ref{lem:stability} we can always find a small closed ball around $A_j$, such that its union with all the other sets is polynomially convex.  

\medskip

If $(H_{n})_{n\geq 1}$ is a sequence of automorphisms of $\C^m$, then for all $0\leq n\leq k$ we denote
$$H_{k,n}:=H_{k}\circ \dots \circ H_{n+1}.$$
Notice that with these notations we have for all $n\geq 0$, $$H_{n+1,n}=H_{n+1}, \quad H_{n,n}={\sf id}.$$
If for all $n\geq 1$ we have $H_{n}(\B)\subset \B $ then we define the {\sl basin} of the sequence $(H_n)$ as the domain
$$
\Omega_H:=\bigcup_{n\geq 0} H_{n,0}^{-1}(\B).
$$
The following lemma was established in \cite[Lemma 2]{ABTP} and will be used in Section 2 to determine the complex structure of a wandering domain.
\begin{lemma}\label{lem:compare}
To every finite family  $(F_1, \dots , F_k)$ of holomorphic automorphisms of $\C^m$  satisfying  $F_{j}(\B)\subset\subset \B $ for all $1\leq j\leq k$
we can associate $\e(F_1, \dots , F_k)>0$ such that the following holds:

Given any two sequences  $(H_{n})_{n\geq 1}$ and $(G_{n})_{n\geq 1}$ of holomorphic automorphisms of $\C^m$ satisfying $H_{n}(\B)\subset\subset \B $ and $G_{n}(\B)\subset \B $ for all $n\geq1$, and moreover satisfying
$$
\|H_{n}-G_{n}\|_{\overline\B}\leq \e(H_1, \dots , H_n),\quad \forall n\geq 1,
$$
the basins $\Omega_{G}$ and $\Omega_H$ are biholomorphically equivalent.
\end{lemma}

In Section 2 and Section 5 we will use the following version of the classical $\lambda$-Lemma  (see \cite[Lemma 7.1]{PM}) for the construction of an oscillating orbit.

\begin{lemma}($\lambda$-Lemma) Let F be a holomorphic automorphism of $\C^m$ with a saddle  fixed  point  at  the  origin.  Denote  by $W^s(0)$ and $W^u(0)$ the stable  and  unstable  manifolds  respectively.  Let $p\in W^s(0)\backslash\{0\}$ and $q\in W^u(0)\backslash\{0\}$, 
and let $D^u(p)$ and $D^s(q)$ be holomorphic polydisks through $p$ and $q$, of dimension $dim W^u$ and $dim W^s$, respectively transverse to $W^s(0)$ and $W^u(0)$. 
 Let $\epsilon>0$.  Then there exists $N\in\N$ and $N_1>2N+1$,and a point $x\in D^u(p)$ with $F^{N_1}(x)\in D^s(q)$  
 such that $\|F^n(x)-F^n(p)\|<\epsilon$ and $\|F^{N_1-n}(x)-F^{-n}(p)\|<\epsilon$ for $0\leq n\leq N$, and $\|F^n(x)\|< \epsilon$ when $N<n<N_1-N$.
\end{lemma}

\section{Oscillating wandering domains} 
The existence  holomorphic automorphisms of $\C^m$, that admit an oscillating wandering domain,  has been proven by Forn\ae ss and Sibony \cite{FS}. Recently constructed examples show that such wandering domains can be biholomorphic to $\C^m$ \cite{ABFP} and to a Short $\C^m$ \cite{ABTP}. In this section we slightly modify these two constructions and reprove them using tools of Anders\'en--Lempert theory. This modification will serve as the basis for the construction of the oscillating wandering ball, which will be presented in the last section of this paper.
The following proposition is a version of \cite[Proposition 4]{ABTP}.

\begin{proposition}\label{propgeneral} Let $(H_k)_{k\geq 1}$ be a sequence of holomorphic automorphisms of $\C^m$  satisfying $H_k(0)=0$ and  $H_{k}(\B)\subset\subset \B $ for all $k\geq 1$ and let $\epsilon_k=\epsilon(H_1,\ldots, H_k)$ be as in Lemma \ref{lem:compare}. 
There exists a sequence  $(F_k)_{k\geq 0}$ of holomorphic automorphisms of $\C^m$,
a sequence of points $(P_k)_{k\geq 0}$, sequences positive real numbers $(\beta_k)_{k\geq 0}\searrow 0$,    $(R_k)_{k\geq 0}\nearrow \infty$,  $(r_k)_{k\geq 0}\nearrow \infty$, strictly increasing sequences of integers  $(n_k)_{k\geq 0}$ and $(N_k)_{k\geq 0}$ satisfying  $n_0=0$ and $N_{k-1}\leq n_k\leq N_k$, and   such that the following properties are satisfied:

\begin{enumerate}[label=(\alph*)]
\item $\B(0,\frac{r_{k-1}}{2})\subset\subset F_k(\B(0,R_k))$ for all $k\geq 1$,
\item $\|F_{k}-F_{k-1}\|_{B(0,R_{k-1})}\leq 2^{-k}$ for all $k\geq 1$,
\item $F_{k}(P_n)=P_{n+1}$ for all $0\leq n< N_k$ and all $k\geq 1$,
\item $\|P_{n_k}\|\leq \frac{1}{k}$ for all $k\geq 1$,
\item $\|P_{N_{k}}\|> R_{k}$ for all $k\geq 1$,
\item for all $k\geq 1$ we have $\beta_j<\frac{1}{k+1}$ for $N_k<j\leq  N_{k+1}$.
\item for all $0\leq s\leq k$,
$$F_{k}^j(\B(P_{n_s},\beta_{n_s}))\subset \subset\B(P_{n_s+j},\beta_{n_s+j}), \quad \forall\,1\leq j\leq N_{k}-n_s,$$
\item  for all $ 1\leq s\leq k$,
$$
\|\Phi_{n_{s}}^{-1}\circ F_k^{n_{s}-n_{s-1}}\circ\Phi_{n_{s-1}}-H_{s}\|_{\overline{\B}}<\epsilon_s,
$$
where   $\Phi_n(z):=P_n+\beta_n z$.
\end{enumerate}
\end{proposition}
\begin{remark} In the  above proposition, the properties $(a)$ and $(b)$ imply that the sequence $F_k$ converges uniformly on compacts to an automorphism $F$. The properties $(c)$--$(g)$ ensure the existence of an oscillating  wandering domain for $F$ and the property $(h)$ determines the intrinsic dynamics and the geometry of such a domain.  
\end{remark}

\begin{proof}
We prove this proposition by induction on $k$.   
\medskip

{\bf Base case:} We start the induction by letting 
$$F_0(z_1,\ldots, z_m)=(\frac{1}{2}z_1,\ldots,\frac{1}{2}z_\iota,2z_{\iota+1},\ldots,2z_m)$$ for some $1\leq \iota<m$. Let  $R_0=1$ and let $r_0>1$ be so large that $F_0(\B(0,R_0))\subset\subset \B(0,r_0)$ and set $K_0:=F^{-1}_0(\overline{\B}(0,r_0))$. Moreover we let $n_0=N_0=0$, $\beta_0=1$, and choose any $P_0\in  \mathbb{C}^m\backslash K_0$ such that $K_0$ and $\overline{\B}(P_{0},\beta_{0})$ are disjoint and their union is polynomially convex.
Notice that all conditions are trivially satisfied for $k =0$. 
\medskip

{\bf Induction hypothesis:} Let us suppose that conditions $(a)$---$(h)$ hold for certain $k$ and that for $K_k:=F^{-1}_k(\overline{\B}(0,r_{k}))$ we have:
\begin{itemize}
\item  $\overline{\B}(0,R_{k})\subset K_k$
\item $K_k\cap \overline{\B}(P_{N_k},\beta_{N_k})=\emptyset$ 
\item $K_k\cup \overline{\B}(P_{N_k},\beta_{N_k})$ is polynomially convex.
\end{itemize}
 
\medskip 
{\bf Inductive step:} We proceed with the constructions satisfying the conditions for $k+1$. 
First choose  $R_{k+1}>\|P_{N_k}\|+1$ such that $ K_k\subset\B(0,R_{k+1})$. 
By the $\lambda$-Lemma there exists a finite $F_k$-orbit $(Q_j)_{0\leq j \leq M}$, the new oscillation, which goes inward along  the  stable  manifold  of $F_k$,  reaching  the  ball $\ B(0,\frac{1}{k+1})$,  and  then  outwards  along  the  unstable manifold of $F_k$. That is there exists a finite set of points satisfying  i.e. $F_k^j(Q_0)=Q_j$ for all $0\leq j\leq M$
and:
\begin{enumerate}[label=(\roman*)]
\item[($\mathcal{A}1$)] $\|Q_{j}\|<R_{k+1}$ for all $0\leq j<M$,
\item[($\mathcal{A}2$)] $\|Q_{M}\|>R_{k+1}$,
\item[($\mathcal{A}3$)] $\|Q_{\ell}\|<\frac{1}{k+1}$ for some $0<\ell<M$.
\end{enumerate}

Using approximation Theorem \ref{al} we will construct a  map $F_ {k+1}$ that connects  the  previously  constructed  finite  orbit $(P_j)_{0\leq j\leq N_k}$ with the new oscillation $(Q_j)_{0\leq j\leq M}$.

Next we choose small enough $0<\theta<\frac{1}{k+1}$ so that: 
\begin{enumerate}[label=(\roman*)]
\item[($\mathcal{B}1$)] the ball $\overline{\B}(Q_M, \theta)$ is disjoint from $\overline{\B}(0, R_{k+1})$,

\item[($\mathcal{B}2$)] the balls
\begin{equation*}\overline{\B}(P_{N_k},\beta_{N_k}),\qquad \overline{\B}(Q_0, \theta) \qquad \overline{\B}(Q_M, \theta)
\end{equation*}
are pairwise disjoint, and disjoint from the set
\begin{equation*}
L:=K_k\cup \bigcup_{0< i<M} \overline{\B}(Q_i,\theta),
\end{equation*}
and their union with $L$ is a polynomially convex set (see Lemma \ref{lem:stability}),
\item[($\mathcal{B}3$)]  $\B(P_{N_k},\beta_{N_k})\cup{\B}(Q_0, \theta)\cup L\subset\subset \B(0, R_{k+1})$.
\end{enumerate}

Observe that by continuity of $F_k$ there exists
$0<s_\ell<\theta$ small enough such that 
 $$F_k^j(\B(Q_\ell,s_\ell))\subset \subset \B(Q_{\ell+j},\theta)\qquad \text{for all } 0\leq j\leq M-\ell,$$ and such that
\begin{equation*}
F_k^{-j}(\B(Q_\ell,s_\ell))\subset\subset \B(Q_{\ell-j}, \theta)\quad \text{for all } 0\leq j\leq \ell.
\end{equation*}

\medskip

We are now ready to construct the map $F_{k+1}$. Set $n_{k+1}:=N_k+\ell$,  $P_{n_{k+1}}:=Q_\ell$ and  $\beta_{n_{k+1}}:=s_\ell$. 

 First we define automorphisms  
\begin{equation}\label{phi}
\phi:= F_k^{-\ell+1}\circ\Phi_{n_{k+1}}  \circ H_{k+1}\circ \Phi_{n_{k} }^{-1}\circ F_k^{-(N_k-n_{k})}
\end{equation}
and a compact starshapelike domain 
$$
\mathcal{W}:=F_k^{N_k-n_{k}}(\overline \B(P_{n_{k}},\beta_{n_{k}}))\subset \B(P_{N_k},\beta_{N_k}).
$$
Observe that the following properties hold:
\begin{enumerate}[label=(\roman*)]
\item[($\mathcal{C}1$)] $\phi(P_{N_k})=Q_1$,
\item[($\mathcal{C}2$)] $\phi(\mathcal{W})\subset\subset F_k( \B(Q_{0}, \theta))$,
\item[($\mathcal{C}3$)] $F_k^j(\phi(\mathcal{W}))\subset\subset \B(Q_{j+1}, \theta)$ for all $0\leq j<\ell-1$,
\item[($\mathcal{C}4$)] $F_k^{\ell-1}(\phi(\mathcal{W}))\subset\subset \B(Q_{\ell}, s_{\ell})$.
\end{enumerate}
In the terminology of Theorem \ref{al} let 
\begin{equation*}
A_1=L\cup  \overline{\B}(Q_M, \theta),\qquad A_2= \mathcal{W}.
\end{equation*} 
It follows from the condition $(\mathcal{B}2)$ and $\mathcal{W}\subset \B(P_{N_k},\beta_{N_k})$ that these two sets are pairwise disjoint and that their union is polynomially convex. 
Next we set $q_1(z)=F_k(z)$ on $A_1$ and $q_2(z)=\phi(z)$ on $A_2$.  We claim that their images 
$B_1=q_1(A_1)$ and $B_2=q_2(A_2)$ are pairwise disjoint and that their union is also polynomially convex. 
First recall that by $(\mathcal{C}2)$ we have  $B_2=\phi(\mathcal{W})\subset\subset F_k( \B(Q_{0}, \theta))$ and  from the condition $(\mathcal{B}2)$ we know that $A_1$ and $\B(Q_{0}, \theta)$ are disjoint and their union is polynomially convex. Since $F_k$ is an automorphism we can conclude that  $B_1=F_k(A_1)$ and  $F_k( \B(Q_{0}, \theta))$ are also disjoint and their union  is polynomially convex. Finally since $\phi$ is an automorphism and $B_2=\phi(\mathcal{W})\subset\subset F_k( \B(Q_{0}, \theta))$ is a starshapelike it follows that $B_1$ and $B_2$ are pairwise disjoint and that their union is also polynomially convex. 
\medskip

 By Theorem \ref{al} there exists an automorphism $g_{k+1}$ such that

\begin{enumerate}[label=(\Roman*)]
\item $\|g_{k+1}-F_k\|\leq \delta_k$ on the set $L\cup\overline\B(Q_M,\theta)$,
\item $\|g_{k+1}-\phi\|\leq \delta_k,$ on $\mathcal{W}$
\item $g_{k+1}(0)=0$ and $d_0g_{k+1}=d_0F_k$
\item $g_{k+1}(P_j)=F_k(P_j)$ for all $0\leq j<P_{N_k}$,
\item $g_{k+1}(P_{N_k})=Q_1$ and $g_{k+1}(Q_j)=F_k(Q_j)$ for all $0<j\leq M$,
  
 \end{enumerate} 
where we have chosen  $\delta_k\leq \frac{1}{2^k}$ small enough such that:
\begin{enumerate}[label=(\roman*)]
\item $g_{k+1}^j(\mathcal{W})\subset  \B(Q_{j}, \theta),$ for all    $0< j< \ell,$
\item $g_{k+1}^\ell(\mathcal{W})\subset  \B(Q_{\ell}, s_\ell),$ 
\item $g_{k+1}^j(\B(Q_{\ell}, s_\ell))\subset\subset  \B(Q_{\ell+j},\theta),$ for all $0\leq j\leq M-\ell,$
\item for all $0\leq s\leq k$, 
$$g_{k+1}^j(\B(P_{n_s},\beta_{n_s}))\subset \subset\B(P_{n_s+j},\beta_{n_s+j}), \quad \forall 1\leq j\leq N_{k}-n_s.$$
\item  $ \|\Phi_{n_{s+1}}^{-1}\circ g_{k+1}^{n_{s+1}-n_{s}}\circ\Phi_{n_{s}}-H_{s+1}\|_{\overline{\B}}<\varepsilon_{s}$ for all $0\leq s\leq k$
\end{enumerate} 
\medskip 

\begin{figure}[t]
\centering
\label{fig:Detour}
\includegraphics[width=4in]{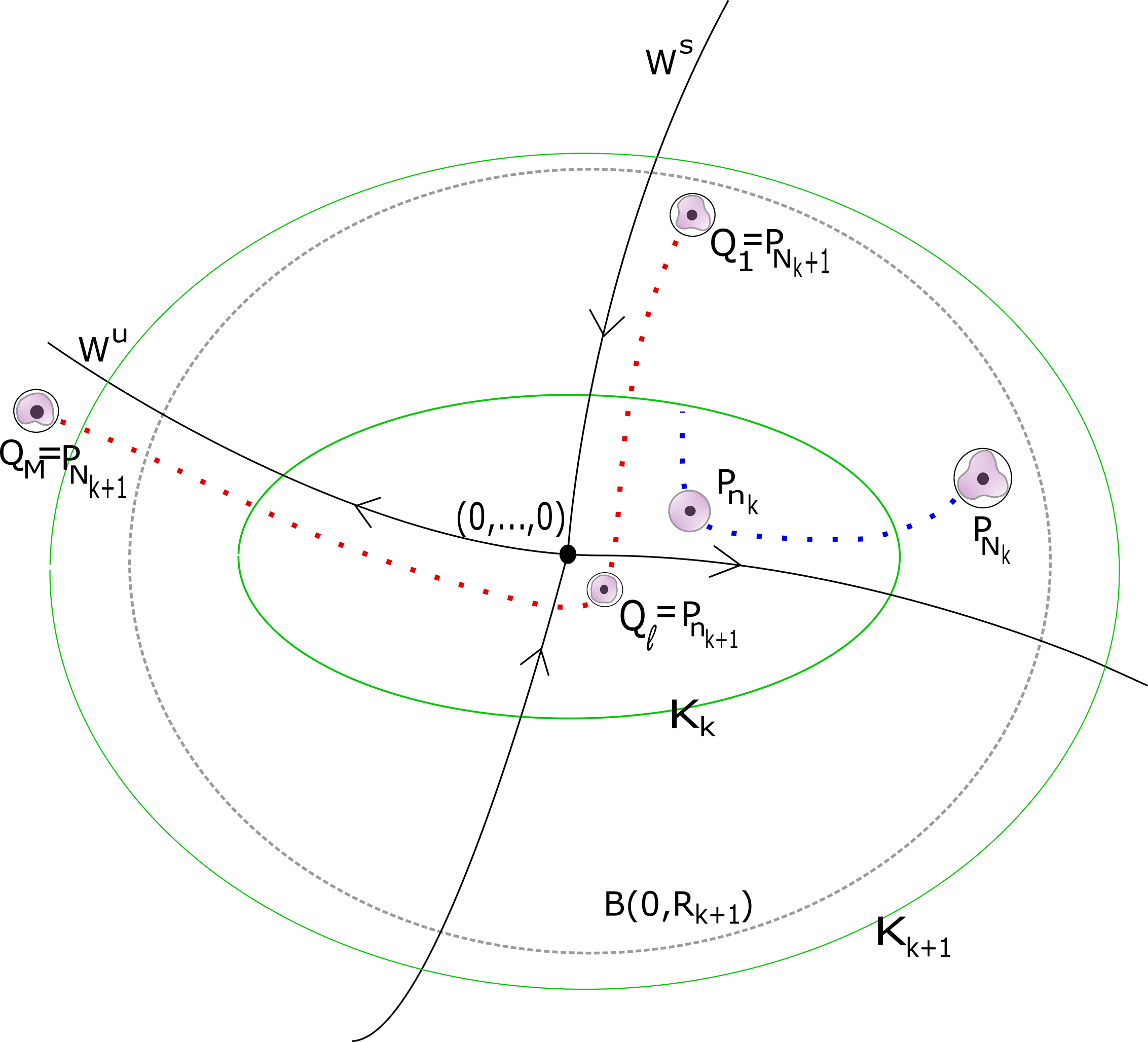}
\caption{A sketch of the piece of orbit constructed at step $k+1$. In red, the points $Q_n$  constructed using the Lambda Lemma. The circles around points represent balls $\B(P_{n_k},\beta_{n_k})$, $\B(P_{N_k},\beta_{N_k})$, $\B(Q_\ell,s_\ell)$, $\B(Q_1,\theta)$ and $\B(Q_M,\theta)$. The colored domains inside balls are $F_k$-images of the set $\mathcal{W}$, see properties $(\mathcal{C}1)-(\mathcal{C}4$) }
\end{figure}

\begin{remark} Note that the map $g_{k+1}$ already satisfies all the properties $(b)$---$(h)$ of the Proposition \ref{propgeneral}. In order  to make sure that also property $(a)$ is satisfied, we need to post-compose our map $g_{k+1}$ with an appropriate automorphism. Note that this property is needed to ensure that the limit map of the sequence $(F_k)$ is surjective.
\end{remark}
\medskip

Let us continue with induction by choosing $r_{k+1} >r_k+1$  such that 
$$g_{k+1} (\B(0, R_{k+1})) \subset\subset \B(0, r_{k+1}).$$
 Since compact sets  $\overline{\B}(Q_{M},\theta)$ and $\overline{\B}(0, R_{k+1})$  are disjoint starshapelike domains whose union is polynomially convex  the same holds for  
 $$
 U:=g_{k+1}(\overline{\B}(Q_{M},\theta)), \qquad V:=g_{k+1}(\overline{\B}(0, R_{k+1})).
 $$
  Choose any point $Q'\in \C^m$  such that the ball $\overline{\B}(Q',\theta)$ lies in the complement of  $\overline{\B}(0, r_{k+1})$ and let $\psi$ be a linear map satisfying $(\psi\circ g_{k+1})(Q_M)=Q'$ and $ \psi(U)\subset {\B}(Q',\theta).$
 
In the terminology of Theorem \ref{al} we set $A_1=U$, $A_2=V$ and we have already seen that these two sets are disjoint and their union is polynomially convex. Set $q_1(z)=\psi(z)$ and $q_2(z)=z$ and let us show that also $B_1=q_1(A_1)=\psi(U)$ and $B_2=q_2(V)=V$ are disjoint and their union is polynomially convex. The set  $B_1=\psi(U)$ is a starshapelike and is a subset of a ball $\overline{\B}(Q',\theta)$ which is disjoint from the ball $\overline{\B}(0, r_{k+1})$. Since $B_2$ is a starshapelike domain and it is contained in $\overline{\B}(0, r_{k+1})$ it follows that $B_1$ and $B_2$ are disjoint and their union is polynomially convex.
By Theorem \ref{al} there exists  an automorphism $h$ such that

\begin{enumerate}[label=(\Roman*)]
\item $\|h-{\rm id}\|_V\leq \delta'_k$,
\item $\|h-\psi\|_U\leq \delta'_k$,
\item  $h(Q_j)=Q_j$ for all $0< j\leq M$,
\item $(h\circ g_{k+1})(Q_M)=Q'$
\item $h(0)=0$, $d_0h={\sf id}$, $h(P_j)=P_j$ for all $1\leq j\leq P_{N_k}$,

 \end{enumerate} 
where we have chosen  $\delta_k'\leq \frac{1}{2^k}$ small enough such that
\begin{enumerate}[label=(\roman*)]
\item $(h\circ g_{k+1})^j(\mathcal{W})\subset  \B(Q_{j}, \theta),$ for all  $0< j< \ell,$
\item $(h\circ g_{k+1})^\ell(\mathcal{W})\subset  \B(Q_{\ell}, s_\ell),$ 
\item $(h\circ g_{k+1})^j(\B(Q_{\ell}, s_\ell))\subset\subset  \B(Q_{\ell+j},\theta),$ for all $0\leq j\leq M-\ell,$
\item for all  $0\leq s\leq k$ 
$$(h\circ g_{k+1})^j(\B(P_{n_s},\beta_{n_s}))\subset \subset\B(P_{n_s+j},\beta_{n_s+j}),\quad \forall 1\leq j\leq N_{k}-n_s.$$ 

\item  $ \|\Phi_{n_{s+1}}^{-1}\circ (h\circ g_{k+1})^{n_{s+1}-n_{s}}\circ\Phi_{n_{s}}-H_{s+1}\|_{\overline{\B}}<\varepsilon_{s}$ for all $0\leq s\leq k$
\item $(h\circ g_{k+1})(\overline{\B}(Q_{M},\theta))\subset{\B}(Q',\theta)$.
\end{enumerate}

\medskip

We define $F_{k+1}:=h\circ g_{k+1}$, so that the sequences of points
$$
(P_j)_{0\leq j\leq N_k},\qquad (Q_j)_{0< j\leq M}
$$
together form the start of an $F_{k+1}$-orbit, t.i. $F_{k+1}^j(P_0)=P_j$ for $j\leq N_k+M$ where $P_{N_k+j}=Q_j$ for $ 0< j \leq M$.

Set $N_{k+1}:=N_k+M$ and define  $\beta_j:=\theta$ for $N_k< j< n_{k+1}$ and for $n_{k+1}<j\leq N_{k+1}$. 

Finally we define $K_{k+1}:=F^{-1}_{k+1}(\B(0,r_{k+1}))$ and observe that  $\B(0,R_{k+1})\subset K_{k+1}$. Since  $$F_{k+1}(\B(P_{N_{k+1}},\beta_{N_{k+1}}))\subset\subset \B(Q',\theta),$$ where  $P_{N_{k+1}}=Q_M$ and since the set
$\B(0,r_{k+1}) \cup  \B(Q',\theta)$ 
is polynomially convex it follows that:
\begin{itemize}
\item  $\overline{\B}(0,R_{k+1})\subset K_{k+1}$
\item $K_{k+1}\cap \overline{\B}(P_{N_{k+1}},\beta_{N_{k+1}})=\emptyset$ 
\item $K_{k+1}\cup \overline{\B}(P_{N_{k+1}},\beta_{N_{k+1}})$ is polynomially convex.
\end{itemize}
 
\medskip

It is immediate that conditions $(c)$---$(h)$ are satisfied for the $(k+1)$-th step. Condition $(b)$ follows from the fact that $\B(0,R_k)\subset K_k\subset L\subset \B(0,R_{k+1})$ and that $h$ is almost an identity on $g_{k+1}(\B(0,R_{k+1}))$, hence $F_{k+1}=h\circ g_{k+1}$ approximates $F_k$ on $\B(0,R_k)$ as close as we want. For condition $(a)$
it is enough to prove that  $\B(0,\frac{r_k}{2})\subset F_{k+1}(K_k)$ where 
$K_k=F_k^{-1}(\B(0,r_k))\subset\B(0,R_{k+1})$. Since $g_{k+1}$ is almost $F_k$ on $K_k$ it follows that  
$\B(0,\frac{2r_k}{3})\subset\subset g_{k+1}(F_k^{-1}(\B(0,r_k))$ and since $h$ is almost an identity on $g_{k+1}(K_n)$  it follows that $\B(0,\frac{r_k}{2})\subset(h\circ g_{k+1})(\B(0,R_{k+1})) $. This concludes the inductive step.
\end{proof}

 \begin{remark}\label{rem1} Notice that the map $F_k$ depends on maps $H_{1},\ldots, H_k$, but not on $H_{k+1}$. This means that although the Proposition \ref{propgeneral} starts with a given sequence $(H_k)_{k\geq1}$, we actually have a total freedom of choosing the map $H_{k+1}$ in the $k+1$-th of the induction. In particular we can assume that $H_{k+1}$ satisfies certain conditions coming from the map $F_k$.  
 For example note that in the $k+1$-th step of the induction we chose the radius $\beta_{n_{k+1}}:=s_\ell$,  before  we make an $\epsilon_{k+1}$-approximation to $H_{k+1}$, that is before  we introduce the map $\phi$  in \eqref{phi}  and construct the map  $g_{k+1}$, hence we can assume that the map $H_{k+1}$ satisfies $\|H_{k+1}(z)\|\leq \beta_{n_{k+1}}^{k+1}\|z\|$. In particular this means that instead of starting with a sequence $(H_k)_k$ we can choose its elements inductively so that they satisfy $\|H_{k}(z)\|\leq \beta_{n_{k}}^{k}\|z\|$ for all $k$. 
 \end{remark}
 
\medskip
Now we show that Proposition \ref{propgeneral} implies the existence of an oscillating wandering Fatou component. The arguments in the following two paragraphs are similar to those in the proof of \cite[Theorem 1]{ABTP} but we choose to present them here, so that the paper remains self-contained. 
 \medskip
 
 Let $(F_k)$ be a sequence of automorphisms of $\mathbb{C}^m$ satisfying conditions $(a)-(h)$ of  Proposition \ref{propgeneral}. The sequence $(F_k)$ converges uniformly on compact subsets to  an automorphism $F$ of $\C^m$ with an isolated fixed point at the origin and  with $d_0F$ being a diagonal matrix with eigenvalues equal to $\frac{1}{2}$ and $2$. There is an unbounded orbit $(P_n)$, a  sequence
 $\beta_n\to 0$ and  a strictly increasing sequences of integers  $(n_k)$ and $(N_k)$ such that the following properties are satisfied:
\begin{enumerate}[label=(\roman*)]
\item $P_{n_k}\to 0$ and $P_{N_k}\to \infty$ as $k\rightarrow \infty$,
\item $F^j(P_0)=P_j$ for all $j\geq0$,
\item for all $k\geq 0$,
\begin{equation}\label{palle}
F^j(\mathbb{B}(P_{n_k},\beta_{n_k}))\subset \mathbb{B}(P_{n_k+j},\beta_{n_k+j}),\quad \forall\, j\geq 0.
\end{equation}
\item  if for all  $k\geq 1$ we denote
  $$G_{k}:=\Phi_{n_{k}}^{-1}\circ F^{n_{k}-n_{k-1}}\circ\Phi_{n_{k-1}}\in {\rm Aut}(\C^m),$$
 then by combining conditions (g) and (h) from Proposition \ref{propgeneral} it follows that $G_{k}(\B)\subset \B$ for all $k$, and
 \begin{equation}\label{saruman}
 \|G_{k}-H_{k}\|_{\overline{\B}}\leq\varepsilon(H_1,\dots, H_k),\quad \forall\,k\geq 1.
 \end{equation}
 
\end{enumerate}
\medskip 
By  (\ref{palle}) the Euclidean diameter of $F^j(\mathbb{B}(P_{n_k},\beta_{n_k}))$ is bounded for all  $j\geq 0$, hence each ball $\B(P_{n_k},\beta_{n_k})$ is contained in some Fatou component which we denote by $\mathcal{F}_{n_k}$. Let us show that the above properties imply the existence of a wandering Fatou component.

\begin{lemma}If $i\neq j$ then  $\mathcal{F}_{n_i}\neq \mathcal{F}_{n_j}$, and hence $F$ has an oscillating wandering Fatou component. 
\end{lemma}
\begin{proof}Since $\beta_n\to 0$, by (\ref{palle}) it follows that all limit functions on each $\mathcal{F}_{n_j}$ are constants. 
We claim that $\mathcal{F}_{n_i}\neq \mathcal{F}_{n_j}$ if and only if $j\neq i$, which together with $(i)$ and $(ii)$ implies that they are all oscillating wandering domains.
Assume by contradiction that $\mathcal{F}_{n_j}=\mathcal{F}_{n_i}$, and set $k:=n_i-n_j$.
Since the origin is a fixed point and  $d_0F$ is diagonal matrix with eigenvalues equal to $\frac{1}{2}$ and $2$, there exists a neighbourhood  $U$  of the origin that contains no periodic points of order less than or equal to $k$.
 Since the sequence of points $(P_n)$ oscillates, there exists a  subsequence  $(P_{m_\ell})$  such that $P_{m_\ell}\to z \in U \setminus \{0\}.$ But then
$$F^{m_\ell-n_j}(P_{n_i})=F^{n_i-n_j}(P_{m_\ell})\to F^k(z)\neq z,$$
which contradicts $F^{m_\ell-n_j}(P_{n_j})=P_{m_\ell}\to z$.

This completes the proof of the existence of an oscillating wandering domain. 

\end{proof}
%

\medskip

Let $\mathcal{F}_0$ be an oscillating wandering domain of $F$ containing $\B(P_{0},\beta_{0})$. The following theorem tells us how we can use the sequence $(H_k)$ to determine the complex structure of $\mathcal{F}_0$. 

\begin{theorem}\label{fat} Let $F$ and $(\beta_{n_{k}})_{k\geq 0}$ be as above. If the sequence of automorphisms $(H_{k})_{k\geq 1}$ satisfies $\|H_{k}(z)\|\leq \beta_{n_{k}}^{k}\|z\|$ on $\B$ for all $k\geq 1$  then the oscillating wandering Fatou component $\mathcal{F}_0$ is biholomorphic to $\Omega_H$.
\end{theorem}
\begin{proof}First note that such sequence of $(H_k)_{k\geq 1}$ exists due to Remark \ref{rem1}.  
We define 
$$
\Omega_F:=\bigcup_{k=0}^{\infty} F^{-n_k}(\mathbb{B}(P_{n_k},\beta_{n_k})).
$$
Since all balls $\mathbb{B}(P_{n_k},\beta_{n_k})$ belong to the Fatou set 
it follows that  $\Omega_F\subset \mathcal{F}_0$.  Recall that
$\Phi_{0}(z)=P_0+\beta_0\cdot z$ and observe that $\Omega_F=\Phi_{0}(\Omega_G)$.

By \eqref{saruman} and  Lemma \ref{lem:compare} we know  that the basins $\Omega_H=\bigcup_{k\geq 0} H_{k,0}^{-1}(\B)$ and $\Omega_G=\bigcup_{k\geq 0} G_{k,0}^{-1}(\B)$ are  biholomorphic, hence $\Omega_F$ is  biholomorphic to $\Omega_H$.  It remains to prove that $\Omega_F=\mathcal{F}_0$. 

This will be done using the plurisubharmonic method which was introduced in \cite{ABFP} and further developed in  \cite[Theorem 7]{ABTP}. The idea of the proof is the following. We construct a plurisubharmonic function $\psi^*$ that is defined on $\mathcal{F}_0$ and it satisfies $\psi^*\equiv 0$ on $\mathcal{F}_0\backslash \Omega_F$ and $\psi^*\leq -1$  on $\Omega_F$. Due to the maximum principle the existence of such function is possible if and only if $\mathcal{F}_0\backslash \Omega_F=\emptyset$.
\medskip

First observe that, by taking smaller $\epsilon_k=\epsilon(H_1,\ldots,H_k)$ in the Proposition \ref{propgeneral} if necessary, we may assume that $$\|G_k(z)\|\leq 2\|H_k(z)\|$$ for all $z\in\B$ and all $k\geq1$. 

Define a sequence of plurisubharmonic functions
$$
\Psi_k(z) :=  \frac{\log \| F^{n_k}(z)-P_{n_k}\|}{-k \log \beta_{n_k}}.
$$
Let $\Psi = \limsup_{k\to \infty} \Psi_k $ on $\mathcal{F}_0$ and let  $\Psi^\star$ be its upper semi-continuous regularization (see \cite{Klimek}), hence $\Psi^\star$ is a plurisubharmonic function on $ \mathcal{F}_0$. 
Since the sequence $(P_{n_k})$ is bounded, it follows that for all compact subsets $K\subset \mathcal{F}_0$, we have $\| F^{n_k}(z)-P_{n_k}\|\rightarrow 0$. This implies that $\Psi\leq0 $ on $\mathcal{F}_0$, and hence $\Psi^\star\leq0 $ on $\mathcal{F}_0$.

Recall that $\Omega_F=\Phi_0(\Omega_G)$ and observe that
\begin{align*}
\Psi_k(z)&=\frac{\log \| \Phi_{n_k}\circ G_k\circ\ldots \circ G_1\circ\Phi_0^{-1}(z)-P_{n_k}\|}{-k \log \beta_{n_k}}\\
&=\frac{\log \| G_k\circ\ldots \circ G_1\circ\Phi_0^{-1}(z)\|+\log \beta_{n_k}}{-k \log \beta_{n_k}}\\
&=\frac{\log \| G_k\circ\ldots \circ G_1\circ\Phi_0^{-1}(z)\|}{-k \log \beta_{n_k}}-\frac{1}{k}
\end{align*}

Given any $z\in \mathcal{F}_0$ we know that the sequence $G_k\circ\ldots \circ G_1\circ\Phi_0^{-1}(z)$ is bounded. In particular $G_k\circ\ldots \circ G_1\circ\Phi_0^{-1}(z_0)\in \B(0,1)$ for some $k\geq 1$ if and only if $z_0\in \Omega_F$. This implies that $\Psi(z)=\Psi^{\star}(z)=0$ on $\mathcal{F}_0\backslash \Omega_F$.

On the other hand, let $z_0\in \Omega_{F},$ and let $k\geq 0$ be large enough such that
$z_k:=G_{k-1}\circ\ldots \circ G_1\circ\Phi_0^{-1}(z_0) \in B(0,1)$.
Then 
\begin{align*}
\Psi_k(z_0)=\frac{\log \|G_k(z_k)\|}{-k\log\beta_{n_k}}-\frac{1}{k}&\leq \frac{\log 2\|H_k(z_k)\|}{-k\log\beta_{n_k}}-\frac{1}{k}\\
&\leq \frac{\log 2+k\log\beta_{n_k} }{-k\log\beta_{n_k}}-\frac{1}{k}\\
&\leq -1+\frac{\log 2 }{-k\log\beta_{n_k}}-\frac{1}{k}.
\end{align*}
It follows that $\Psi(z)\leq -1$ for all $z \in \Omega_{F}$, which implies that  $\Psi^\star(z)\leq -1$ for all $z \in \Omega_{F}$. Since $\mathcal{F}_0$ is open and connected, it follows from the maximum principle for plurisubharmonic functions that $\mathcal{F}_0 \setminus \Omega_{F}$ must be empty, which completes the proof.
\end{proof}

{\bf Example 1.} For $H_k(z)=\beta_{n_k}^k z$ it is easy to see that $\Omega_H=\C^m$ and therefore by Theorem \ref{fat} the oscillating wandering Fatou component $\mathcal{F}_0$ is biholomorphic to $\C^m$. 
\medskip

{\bf Example 2.} Let 
$$
H_k(z_1,\ldots,z_m)=\left(\left(\frac{z_1}{2}\right)^{d_k}+2^{-d_k\cdots d_1}z_m, 2^{-d_k\cdots d_1}z_1,\ldots,2^{-d_k\cdots d_1} z_{m-1}\right)
$$
 where integers $d_k>0$ are chosen so that 
$\|H_{k}(z)\|\leq \beta_{n_{k}}^{k}\|z\|$ on $\B$ for all $k\geq 1$. By \cite[Proposition 3]{ABTP}, which is a slightly  modified version of \cite[Theorem 1.4.]{F} we know that $\Omega_H$  is a Short $\C^m$, and therefore by Theorem \ref{fat} the oscillating wandering Fatou component $\mathcal{F}_0$ is also a Short $\C^m$. Such a domain supports a non-constant bounded plurisubharmonic function which implies that $\mathcal{F}_0$ is not biholomorphic to $\C^m$.

\section{Escaping wandering Ball}
In this section we prove Theorem \ref{thmescaping} using the following proposition. 

\begin{proposition}\label{propescape}There exists a sequence $(F_k)_{k\geq 0}$ of holomorphic automorphisms of $\C^m$, disjoint sequences of points $(P_n)_{n\geq 0}$,  $(T^j_n)_{j \geq 1, n \geq 0 }$, $(S_j)_{j\geq1}$ and sequences positive real numbers   $(R_k)_{k\geq 0}\nearrow \infty$ and  $(r_k)_{k\geq 0}\nearrow \infty$,  such that the following properties are satisfied:

\begin{enumerate}[label=(\alph*)]
\item $\B(0,\frac{r_{k-1}}{2})\subset\subset F_k(\B(0,R_k))$ for all $k\geq 1$,
\item   $\|F_{k}-F_{k-1}\|_{B(0,R_{k-1})}\leq 2^{-k}$ for all $k\geq 1$,
\item $\|P_{k}\|> R_k$ for all $k\geq 0$,
\item $F_{k}^j(\B(P_{0},1))\subset \subset\B(P_{j},2),$ for all $1\leq j\leq k$  and all $ k\geq 0$.
\item points $T^j_0$ accumulate densely on the $b\B(P_0,1)$.
\item $F_k(T_n^j)=T_{n+1}^j$ for all $ 1\leq j\leq k$ and all $0\leq n\leq j-1$
\item $T_{j}^j\in \B(S_j,1)$ for all $j\geq 1$.
\item $F_k(S_j)=S_j$ for all $1\leq j\leq k$.
\item $\|F_k(z)-S_j\|\leq \frac{k}{2(k+1)}\|z-S_j\|$ on $\overline{\B}(S_j,1)$ for all $1\leq j\leq k$.
\end{enumerate}
\end{proposition}

Before proving this proposition, let us show that it implies the existence of an escaping wandering Fatou component which is the unit ball.  

\subsection{Proof of Theorem \ref{thmescaping}}Let $(F_k)_{k\geq 0}$ be a sequence of holomorphic automorphisms of $\C^m$ given by Proposition  \ref{propescape}. This sequence converges uniformly on compacts to a holomorphic automorphism $F$ of $\C^m$. Moreover there exist disjoint sequences of points $(P_n)_{n\geq 0}$,  $(T^j_0)_{j \geq 1}$, $(S_j)_{j\geq1}$ and strictly increasing sequence of positive real numbers  $(R_k)_{k\geq 0}$  such that the following holds:
\begin{enumerate}
\item $\|P_{k}\|> R_k$ for all $k\geq 0$,
\item $F^k(\B(P_0,1))\subset \subset\B(P_k,2)$ for all $k\geq 0$,
\item $F(S_j)=S_j$ and $\|F(z)-S_j\|\leq\frac{1}{2}\|z-S_j\|$ on $\overline{\B}(S_j,1)$  for all $j\geq 1$
\item Points $T^j_0$ accumulate densely on the $b\B(P_0,1)$ and $F^k(T^j_0)\rightarrow S_j$ as $k\rightarrow\infty$,
\end{enumerate}
Observe that properties $(1)$ and $(2)$ follow from properties $(c)$ and $(d)$ of the proposition. These properties imply that the forward orbit of every point in $\B(P_0,1)$ eventually leaves every compact set. Furthermore since the Euclidean  diameter of $F^j(\B(P_0,1))$ is bounded for all $j\geq 0$ it follows that $\B(P_0,1)$ is contained in some Fatou component $\mathcal{F}_0$. It remains to prove that $\mathcal{F}_0 = \B(P_0,1)$.

Observe that properties  $(3)$ and $(4)$ follow from properties $(e)-(i)$ of the proposition. Namely points $S_j$ are attracting fixed point of the map $F$ and all the points in the ball  $\B(S_j,1)$ converge towards $S_j$ under the iteration. Moreover $F$ was constructed in such a way that the $F$-orbit of any $T^j_0$ eventually lands in one of balls $\B(S_j,1)$, and hence it converges to $S_j$. It follows that each $T^j_0$ is contained in the attracting Fatou component of $S_j$.  

Notice that since  $\mathcal{F}_0$ is an  open set, the condition $\mathcal{F}_0 \neq \B(P_0,1)$ would imply that there exists $j>0$ so that $T_0^j\in \mathcal{F}_0$. 
But this would contradict the normality of the iterates $(F^n)$ on $\mathcal{F}_0 $, since we know that on the small neighbourhood of $T_0^j$ the sequence of iterates converges to the constant $S_j$ but in the same time the sequence of iterates compactly diverges on $\B(P_0,1)$. $\square$

\subsection{Proof of Proposition \ref{propescape}} We prove this proposition by induction on $k$.

\medskip
{\bf Base case:}  We start the induction by letting $F_0=\rm{id}$,  $r_0=2$ and $R_0=1$. We define $K_0=F^{-1}_0(\overline{\B}(0,r_{0}))=\overline{\B}(0,2)$. Finally we choose a point  $P_0$ satisfying $\|P_0\|>4$ and a sequence $(T_0^j)_{j\geq 1}\subset \B(P_0,2)\backslash \overline{\B}(P_0,1)$ which  accumulates densely on $b\B(P_0,1)$ and for which the sequence of distances  $\|T_0^j-P_0\|\rightarrow 1$ is strictly decreasing. In this setting all conditions $(a)$---$(i)$ are satisfied for $k =0$.

\medskip

{\bf Induction hypothesis:} Let us suppose that conditions $(a)$---$(i)$ hold for certain $k$ and that for $K_k:=F^{-1}_k(\overline{\B}(0,r_{k}))$ we have:
\begin{itemize}
\item  $\overline{\B}(0,R_{k})\subset K_k$
\item $K_k\cap \overline{\B}(P_{k},2)=\emptyset$ 
\item $K_k\cup \overline{\B}(P_{k},2)$ is polynomially convex.
\end{itemize}
 
\medskip 
{\bf Inductive step:}
We proceed with the constructions satisfying the conditions for $k+1$. 

First we choose a point $S_{k+1}$ so that:
\begin{enumerate}[label=(\roman*)]
\item[($\mathcal{A}1$)]  The ball $\overline{\B}(S_{k+1},1)$ is disjoint from the sets $K_k$, $F_k(K_k)$ and $\overline{\B}(P_k,2)$
\item[($\mathcal{A}2$)] $\overline{\B}(S_{k+1},1)\cup K_k\cup \overline{\B}(P_k,2)$ is  polynomially convex.
\item[($\mathcal{A}3$)] $\overline{\B}(S_{k+1},1)\cup F_k(K_k)$ is  polynomially convex.
\end{enumerate}
Next choose $R_{k+1}>R_k+1$ and point $P_{k+1}$ so that:

\begin{enumerate}[label=(\roman*)]
\item[($\mathcal{B}1$)]  $\overline{\B}(S_{k+1},1)\cup K_k\cup  F_k(K_k)\cup \overline{\B}(P_k,2) \subset \B(0,R_{k+1})$
\item[($\mathcal{B}2$)]$\overline{\B}(P_{k+1},2)$ and $\overline{\B}(0,R_{k+1})$ are disjoint
\end{enumerate}

 Finally choose $1<\rho_{k+1}<2$ such that 
 $$
 \|T_0^{k+2}-P_0\|<\rho_{k+1}<\|T_0^{k+1}-P_0\|
 $$ and define a starshapelike compact set 
 $$
 \mathcal{W}:=F_k^{k}(\overline \B(P_{0},\rho_{k+1}))\subset \B(P_{k},2).
 $$
 In the terminology of Theorem \ref{al} we define 
\begin{equation*}
A_1:=K_k,\qquad A_2:= \overline{\B}(S_{k+1},1),\qquad A_3:= \mathcal{W},\qquad A_4:= T_k^{k+1}.
\end{equation*} 
It follows from our construction that all these sets are pairwise disjoint and that their union is polynomially convex. Also note that all these sets are all starshapelike.
Next we define $q_1(z):=F_k(z)$ on $A_1$, $q_2(z):= \frac{(k+1)}{2(k+2)+1}(z-S_{k+1})+S_{k+1}$ on $A_2$, $q_3(z)=\frac{z-P_k}{2}+P_{k+1}$ on $A_3$ and $q_4(z)=z-T_k^{k+1}+Q_0+\frac{2}{3}$ on $A_4$.  Observe that their images 
$B_j:=q_j(A_j)$ where $1\leq j  \leq4$ are  pairwise disjoint and their union is polynomially convex. Moreover we have $B_3\subset\B(P_{k+1},2)$.

Note that by the inductive assumption all balls $\overline{\B}(S_{j},1)$ for $1\leq j\leq k$ are contained in the compact set $K_k$.  By Theorem \ref{al} there exists an automorphism $g_{k+1}$ such that

\begin{enumerate}[label=(\Roman*)]
\item $\|F_k-g_{k+1}\|\leq \delta_k$ on $K_k$
\item $\|q_3-g_{k+1}\|\leq \delta_k$ on $\mathcal{W}$.
\item $g_{k+1}(S_j)=S_j$  for all $1\leq j\leq k+1$
\item $\|g_{k+1}(z)-S_{k+1}\|\leq \frac{k+1}{2(k+2)+\frac{1}{k+1}}\|z-S_{k+1}\|$ on $\overline{\B}(S_{k+1},1)$,
\item $g_{k+1}(T_n^j)=F_k(T_n^j)$ for all $1\leq j \leq k$ and $0 \leq n\leq j-1$,
\item $g_{k+1}(T_n^{k+1})=T_{n+1}^{k+1}$ for all $0 \leq n< k$ where  $T_{n}^{k+1}:=F_k^n(T_0^{k+1})$
\item $T_{k+1}^{k+1}:=g_{k+1}(T_{k}^{k+1})\in \B(S_{k+1},1)\backslash \overline{\B}(S_{k+1},\frac{1}{2})$ 

 \end{enumerate} 
where we have chosen  $\delta_k\leq \frac{1}{2^k}$ small enough such that:
\begin{enumerate}[label=(\roman*)]
\item $g_{k+1}^j(\B(P_0,\rho_{k+1}))\subset\subset  \B(P_j,2)$, for all $1\leq j\leq k+1$
\item $\|g_{k+1}(z)-S_{j}\|\leq \frac{k+1}{2(k+2)+\frac{1}{k+1}}\|z-S_{j}\|$ on $\overline{\B}(S_{j},1)$ for all $1\leq j\leq k$,
\end{enumerate}
\medskip 
At this point the automorphism $g_{k+1}$ already satisfies properties $(b)$--$(i)$ of the proposition and we continue similarly as in the proof of Proposition \ref{propgeneral}. We choose $r_{k+1} >r_k+1$  so that 
$$g_{k+1} (\B(0, R_{k+1})) \subset\subset \B(0, r_{k+1}).$$
 Since compact sets  $\overline{\B}(P_{k+1},2)$ and $\overline{\B}(0, R_{k+1})$  are disjoint starshapelike domains whose union is polynomially convex  the same holds for their images
 $$U:=g_{k+1}(\overline{\B}(P_{k+1},2)), \qquad V:=g_{k+1}(\overline{\B}(0, R_{k+1})).$$
  Let $Q'\in \C^2$ be a point  such that the ball $\overline{\B}(Q',\theta)$ lies in the complement of  $\overline{\B}(0, r_{k+1})$. Also let $\psi$ be a linear map satisfying $ \psi(U)\subset {\B}(Q',\theta).$
 
  By Theorem  \ref{al} there exists  an automorphism $h$ such that

\begin{enumerate}[label=(\Roman*)]
\item $\|h-{\rm id}\|_V\leq \delta'_k$,
\item $\|h-\psi\|_U\leq \delta'_k$,
\item $h(S_j)=S_j$ for all $0< j\leq k+1$,
\item $h(T_n^{j})=T_n^{j}$ for all $1\leq j\leq k+1$ and $1\leq n\leq j $

 \end{enumerate} 
where we have chosen  $\delta_k'\leq \frac{1}{2^k}$ small enough such that
\begin{enumerate}[label=(\roman*)]
\item $(h\circ g_{k+1})^j(\B(P_0,\rho_{k+1}))\subset\subset  \B(P_j,2)$, for all $1\leq j\leq k+1$
\item $\|(h\circ g_{k+1})(z)-S_{j}\|\leq \frac{k+1}{2(k+2)}\|z-S_{j}\|$ on $\overline{\B}(S_{j},1)$ for all $1\leq j\leq k+1$,
\end{enumerate}

Finally define $K_{k+1}:=F^{-1}_{k+1}(\B(0,r_{k+1}))$ and observe that  $\B(0,R_{k+1})\subset K_{k+1}$. Since  $$F_{k+1}(\B(P_{k+1},2))\subset\subset \B(Q',2),$$ and since the set
$\B(0,r_{k+1}) \cup  \B(Q',\theta)$ 
is polynomially convex it follows that $K_{k+1}$ and $\overline{\B}(P_{k+1},2)$   are disjoint and their union is polynomially convex.

 It is immediate that properties $(c)$---$(i)$ are satisfied for the $(k+1)$-th step. The properties $(a)$ and $(b)$ can be verified by following the last paragraph of the proof of Proposition \ref{propgeneral} verbatim. This concludes the inductive step.$\square$

 \medskip
 
\begin{remark}  We believe that by a slight modification of the above proof, in particular by choosing different map $q_3$,  one can construct examples of wandering balls with different interior dynamics, as it was recently done for transcendental functions in dimension one \cite{BEGRS}. 
\end{remark}

\section{Oscillating wandering Ball}
In the previous two sections we have seen how the tools of  Anders\'en--Lepert theory can be used to construct various examples of oscillating wandering domains and also of the escaping wandering ball. In this section we will see that by combining these two constructions we can construct an oscillating wandering ball. The proof of Theorem \ref{thmoscillating} is based on the following proposition which is a hybrid between Proposition \ref{propescape} and Proposition \ref{propgeneral}.

\begin{proposition}\label{propmain}
There exists a sequence  $(F_k)_{k\geq 0}$ of holomorphic automorphisms of $\C^m$,
disjoint sequences of points $(P_n)_{n\geq 0}$,  $(T^j_n)_{j \geq 1, n \geq 0 }$, $(S_j)_{j\geq1}$ with $(S_j)_{j}$ being  bounded away from the origin, sequences positive real numbers $(\beta_n)_{n\geq 0}\searrow 0$, $(\tau_n)_{n\geq 1}\searrow 0$ ,    $(R_k)_{k\geq 0}\nearrow \infty$,  $(r_k)_{k\geq 0}\nearrow \infty$, strictly increasing sequences of integers  $(n_k)_{k\geq 0}$ and $(N_k)_{k\geq 0}$ satisfying  $n_0=0$ and $N_{k-1}\leq n_k\leq N_k$, and   such that the following properties are satisfied:

\begin{enumerate}[label=(\alph*)]
\item $\B(0,\frac{r_{k-1}}{2})\subset\subset F_k(\B(0,R_k))$ for all $k\geq 1$,
\item $\|F_{k}-F_{k-1}\|_{B(0,R_{k-1})}\leq 2^{-k}$ for all $k\geq 1$,
\item $F_{k}(P_n)=P_{n+1}$ for all $0\leq n< N_k$,
\item $\|P_{n_k}\|\leq \frac{1}{k}$ for all $k\geq 1$,
\item $   \|P_{N_{k}}\|> R_{k}$ for all $k\geq 1$
\item for all $k\geq 1$ we have $\beta_j<\frac{1}{k+1}$ for $N_k<j\leq  N_{k+1}$.
\item $F_{k}^j(\B(P_{0},\beta_{n_0}))\subset \subset\B(P_{j},\beta_j),$ for all $1\leq j\leq N_{k}$  and all $ k\geq 0$.
\item points $T^j_0$ accumulate densely on the $b\B(P_0,1)$.
\item $F_k(T_n^j)=T_{n+1}^j$ for all $ 1\leq j\leq k$ and all $0\leq n\leq N_{j-1}$
\item $T_{N_{j-1}+1}^j\in \B(S_j,\tau_j)$ for all $j\geq 1$.
\item $F_k(S_j)=S_j$ for all $1\leq j\leq k$.
\item $\|F_k(z)-S_j\|\leq \frac{k}{2(k+1)}\|z-S_j\|$ on $\overline{\B}(S_j,\tau_j)$ for all $1\leq j\leq k$.
\end{enumerate}
\end{proposition}
\begin{remark} In the proposition above properties $(a)$ and $(b)$ imply that the sequence $F_k$ converges uniformly on compacts to an automorphism $F$. Properties $(c)$---$(g)$ ensure the existence of an oscillating wandering domain for $F$ and properties $(h)$---$(l)$ ensure that the $F$-orbit of every point $T_0^j$  converges to an attracting fixed point $S_j$.
\end{remark}
 
Before proving this proposition, let us show how it can used to prove our main theorem. 

\subsection{Proof of Theorem \ref{thmoscillating}} Let $(F_k)_{k\geq 0}$ be a sequence of holomorphic automorphisms of $\C^m$ given by Proposition  \ref{propmain}. This sequence converges uniformly on compacts to a holomorphic automorphism $F$ of $\C^m$. Moreover there exists a disjoint sequences of points $(P_n)_{n\geq 0}$,  $(T^j_n)_{j \geq 1, n \geq 0 }$, $(S_j)_{j\geq1}$, sequences positive real numbers $(\beta_n)_{n\geq 0}\searrow 0$,  $(\tau_n)_{n\geq 1}\searrow 0$, strictly increasing sequences of integers  $(n_k)_{k\geq 0}$ and $(N_k)_{k\geq 0}$ such that the following holds:
\begin{enumerate}
\item $F^j(P_0)=P_j$ for all $j\geq 0$,
\item $P_{n_k}\rightarrow 0$ and  $P_{N_k}\rightarrow\infty$ as $k\rightarrow\infty$,
\item $F^j(\B(P_0,1))\subset \subset\B(P_j,\beta_j)$ for all $j\geq 0$,
\item $F(S_j)=S_j$ and $\|F(z)-S_j\|\leq\frac{1}{2}\|z-S_j\|$ on $\overline{\B}(S_j,\tau_j)$  for all $j\geq 1$
\item Points $T^j_0$ accumulate densely on the $b\B(P_0,1)$ and  $F^k(T^j_0)\rightarrow S_j$ as $k\rightarrow\infty$ for all $j\geq 1$.
\end{enumerate}
The remaining argument is similar as in the proof of Theorem \ref{thmescaping}.  By the property $(3)$ the Euclidean  diameter of $F^j(\B(P_0,1))$ is bounded for all $j\geq 0$, hence it follows that $\B(P_0,1)$ is contained in some Fatou component $\mathcal{F}_0$.

 Assume that $\mathcal{F}_0 \neq \B(P_0,1)$. Since  $\mathcal{F}_0$ is an  open set  there exists $j>0$ so that $T_0^j\in \mathcal{F}_0$. Properties $(4)$ and $(5)$ imply that on a small neighbourhood of $T_0^j$ the sequence of iterates $(F^n)$ converges to a constant $S_j$. On the other hand the properties $(2)$ and $(3)$ imply that on  $\B(P_0,1)$ the sequence of iterates $(F^{n_k})$ converges to $0\neq S_j$, hence we are in the contradiction. 
 
Oscillation of the Fatou component $\mathcal{F}_0 = \B(P_0,1)$ follows directly from the properties $(2)$ and $(3)$.  $\square$

\subsection{Proof of Proposition \ref{propmain}}\label{inductionprop}
We prove this proposition by induction on $k$.

\medskip
{\bf Base case:} We start the induction by letting
 $$
F_0(z_1,\ldots, z_m)=(\frac{1}{2}z_1,\ldots,\frac{1}{2}z_\iota,2z_{\iota+1},\ldots,2z_m)
$$ 
for some $1\leq \iota<m$. Let  $R_0=1$ and let $r_0>0$ be so large that $F_0(\B(0,R_0))\subset\subset \B(0,r_0)$. Define $K_0=F_0^{-1}(\overline{\B}(0,r_0))$, $n_0=N_0=0$, $\beta_0=2$, and choose any point $P_0$ so that the sets $\overline{\B}(P_0,\beta_0)$ and $K_0$ are disjoint and their union is polynomially convex.  Finally choose a sequence $(T_0^j)_{j\geq 1}\subset \B(P_0,2)\backslash \overline{\B}(P_0,1)$ which  accumulates densely on $b\B(P_0,1)$ and for which the sequence of distances  $\|T_0^j-P_0\|\rightarrow 1$ is strictly decreasing. Observe that in this setting all conditions $(a)$---$(l)$ are satisfied for $k =0$.

\medskip

{\bf Induction hypothesis:} Let us suppose that conditions $(a)$---$(l)$ hold for certain $k$ and that for $K_k:=F^{-1}_k(\overline{\B}(0,r_{k}))$ we have:
\begin{itemize}
\item  $\overline{\B}(0,R_{k})\subset K_k$
\item $K_k\cap \overline{\B}(P_{N_k},\beta_{N_k})=\emptyset$ 
\item $K_k\cup \overline{\B}(P_{N_k},\beta_{N_k})$ is polynomially convex.
\end{itemize}
 
\medskip 
{\bf Inductive step:}
 Let us proceed with the constructions satisfying the conditions for $k+1$. First let  $R_{k+1}>\|P_{N_k}\|+1$ such that $K_k\subset\B(0,R_{k+1})$. By the $\lambda$-Lemma there exist a finite $F_k$ orbit $(Q_j)_{-1\leq j \leq M}$, i.e. $F_k(Q_{j-1})=Q_j$ for $0\leq j\leq M$, such that:
\begin{enumerate}
\item[$(\mathcal{A}1)$] $\|Q_{j}\|<R_{k+1}$ for all $-1\leq j<M$,
\item[$(\mathcal{A}2)$] $\|Q_{M}\|>R_{k+1}$,
\item[$(\mathcal{A}3)$] $\|Q_{\ell}\|<\frac{1}{k+1}$ for some $0<\ell<M$.

\end{enumerate}
By increasing  $R_{k+1}$ if necessary, we can choose $0<\theta\leq \eta<\frac{1}{k+1}$ so that: 
\begin{enumerate}[label=(\roman*)]
\item[$(\mathcal{B}1)$] the ball $\overline{\B}(Q_M, \theta)$ is disjoint from $\overline{\B}(0, R_{k+1})$,

\item[$(\mathcal{B}2)$] $\B(Q_0, \theta)\subset\subset F_k(\B(Q_{-1}, \eta))$ 
\item[$(\mathcal{B}3)$] the balls
\begin{equation*}\overline{\B}(P_{N_k},\beta_{N_k}),\qquad \overline{\B}(Q_0, \theta), \qquad \overline{\B}(Q_M, \theta),\qquad \overline{\B}(Q_{-1}, \eta) 
\end{equation*}
are pairwise disjoint, and disjoint from the set
\begin{equation}\label{defl}
L:=K_k\cup \bigcup_{0< i<M} \overline{\B}(Q_i,\theta),
\end{equation}
and their union with $L$ is a polynomially convex set  (see Lemma \ref{lem:stability}),
\item[$(\mathcal{B}4)$]  $\B(P_{N_k},\beta_{N_k})\cup{\B}(Q_0, \theta)\cup \B(Q_{-1}, \eta) \cup L\subset\subset \B(0, R_{k+1})$.
\end{enumerate}


By continuity of $F_k$ there exists
$0<s_\ell<\theta$ small enough such that  for all $0\leq j\leq M-\ell$, $$F_k^j(\B(Q_\ell,s_\ell))\subset \subset \B(Q_{\ell+j},\theta)$$ and such that for all $0\leq j\leq \ell$,
\begin{equation}\label{property}F_k^{-j}(\B(Q_\ell,s_\ell))\subset\subset \B(Q_{\ell-j}, \theta).
\end{equation}

Choose $1<\rho_{k+1}<2$ such that $\|T_0^{k+2}-P_0\|<\rho_{k+1}<\|T_0^{k+1}-P_0\|$ and define a starshapelike compact set 
$$\mathcal{W}:=F_k^{N_{k}}(\overline \B(P_{0},\rho_{k+1}))\subset \B(P_{N_k},\beta_{N_k}).$$

 Next define linear automorphisms 
 \begin{equation}\label{phi1}
 \Phi_{1}(z)=\frac{z-P_0}{2\rho_{k+1}},\qquad \Phi_{2}(z)=s_{\ell}\cdot z+Q_\ell
 \end{equation}
   and an automorphism 
\begin{equation}\label{phi2}
\phi= F_k^{-\ell+1}\circ\Phi_{2}\circ\Phi_{1}\circ F_k^{-N_k}.
\end{equation}
Observe that the following holds:
\begin{itemize}
\item[$(\mathcal{C}1)$] $\phi(P_{N_k})=Q_1$,
\item[$(\mathcal{C}2)$] $\phi(W)\subset\subset F_k( \B(Q_{0}, \theta))$,
\item[$(\mathcal{C}3)$] $F_k^j(\phi(\mathcal{W}))\subset\subset \B(Q_{j+1}, \theta)$ for all $0\leq j<\ell-1$,
\item[$(\mathcal{C}4)$] $F_k^{\ell-1}(\phi(\mathcal{W}))\subset\subset \B(Q_{\ell}, s_{\ell})$.
\end{itemize}
Let us write $S_{k+1}:=Q_0$ and $\tau_{k+1}:=\theta$. In the terminology of Theorem  \ref{al} we define 
\begin{equation}\label{A}
A_1:=L\cup  \overline{\B}(Q_M, \theta),\quad A_2:= \overline{\B}(S_{k+1}, \tau_{k+1}),\quad A_3:= \mathcal{W},\quad A_4:= T_k^{k+1}.
\end{equation}
Note that all previously constructed balls $\overline{\B}(S_{j},\tau_{j})$ for $j\leq k$ are contained in the compact set $K_k$ and therefore in $A_1$.  
By the property $(\mathcal{B}3)$ sets $A_j$ are pairwise disjoint and that their union is polynomially convex. Furthermore observe that the sets $A_2$, $A_3$ and $A_4$ are all starshapelike.
Next we define $q_1(z):=F_k(z)$ on $A_1$, $q_2(z):= \frac{(k+1)}{2(k+2)+1}(z-S_{k+1})+S_{k+1}$ on $A_2$, $q_3(z)=\phi(z)$ on $A_3$ and $q_4(z)=z-T_k^{k+1}+Q_0+\frac{2\tau_{k+1}}{3}$ on $A_4$. 

 We claim that their images 
$B_j:=q_j(A_j)$ where $1\leq j  \leq4$ are also pairwise disjoint and their union is polynomially convex. 

First observe that $B_2$ and $B_4$ are disjoint and  contained in $\B(Q_{0}, \theta)\subset F_k(\B(Q_{-1}, \eta))$ and their union is polynomially convex.  Next observe that  \eqref{property} implies $F^{-1}_k(B_3)\subset F_k^{-\ell}(\B(Q_\ell,s_\ell))\subset\subset \B(Q_{0}, \theta)$, hence $B_3\subset F_k(\B(Q_{0}, \theta))$. Our claim now follows directly form the property $(\mathcal{B}3)$. 
\medskip

By Theorem  \ref{al} there exists an automorphism $g_{k+1}$ such that

\begin{enumerate}[label=(\Roman*)]
\item $g_{k+1}(0)=0$ and $d_0g_{k+1}=d_0F_k$,
\item $\|F_k-g_{k+1}\|\leq \delta_k$ on $L\cup\overline\B(Q_M,\theta)$
\item $g_{k+1}(P_j)=F_k(P_j)$ for all $0\leq j<N_k$,
\item $g_{k+1}(Q_j)=F_k(Q_j)$ for all $1\leq j<M$,
\item $g_{k+1}(P_{N_k})=Q_1$
\item $\|\phi-g_{k+1}\|\leq \delta_k$ on $\mathcal{W}$.
\item $g_{k+1}(S_j)=S_j$  for all $1\leq j\leq k+1$
\item $\|g_{k+1}(z)-S_{k+1}\|\leq \frac{k+1}{2(k+2)+\frac{1}{k+1}}\|z-S_{k+1}\|$ on $\overline{\B}(S_{k+1},\tau_{k+1})$,
\item $g_{k+1}(T_n^j)=F_k(T_n^j)$ for all $1\leq j \leq k$ and $0 \leq n\leq N_{j-1}$,
\item for all $0 \leq n< N_{k}$ we have $g_{k+1}(T_n^{k+1})=T_{n+1}^{k+1}$, where  $T_{n}^{k+1}:=F_k^n(T_0^{k+1})$
\item $T_{N_k+1}^{k+1}:=g_{k+1}(T_{N_k}^{k+1})\in \B(S_{k+1},\tau_{k+1})\backslash \overline{\B}(S_{k+1},\frac{\tau_{k+1}}{2})$ 

 \end{enumerate} 
where we have chosen  $\delta_k\leq \frac{1}{2^k}$ small enough such that:
\begin{enumerate}[label=(\roman*)]
\item $g_{k+1}^j(\B(P_0,\rho_{k+1}))\subset\subset  \B(P_j,\beta_j)$, for all $1\leq j\leq N_k$
\item $g_{k+1}^{N_k+j}(\B(P_0,\rho_{k+1}))\subset\subset  \B(Q_{j}, \theta),$ for all    $0< j\leq M$
\item $g_{k+1}^{N_k+\ell}(\B(P_0,\rho_{k+1}))\subset\subset  \B(Q_{\ell}, s_\ell),$ 
\item $\|g_{k+1}(z)-S_{j}\|\leq \frac{k+1}{2(k+2)+\frac{1}{k+1}}\|z-S_{j}\|$ on $\overline{\B}(S_{j},\tau_{j})$ for all $1\leq j\leq k$,
\end{enumerate}

To make sure that the newly constructed automorphism satisfies property $(a)$ of the proposition we need to correct $g_{k+1}$ by pre-composing it with an appropriate automorphism. 

Let  $r_{k+1} >r_k+1$  such that 
$$g_{k+1} (\B(0, R_{k+1})) \subset\subset \B(0, r_{k+1}).$$
 Since the compact sets  $\overline{\B}(Q_{M},\theta)$ and $\overline{\B}(0, R_{k+1})$  are disjoint starshapelike domains whose union is polynomially convex  the same holds for  their images
 $$U:=g_{k+1}(\overline{\B}(Q_{M},\theta)), \qquad V:=g_{k+1}(\overline{\B}(0, R_{k+1})).$$
  Let $Q'\in \C^m$ be a point  such that the ball $\overline{\B}(Q',\theta)$ lies in the complement of  $\overline{\B}(0, r_{k+1})$. Moreover let $\psi$ be a linear map satisfying $(\psi\circ g_{k+1})(Q_M)=Q'$ and
 $$ \psi(U)\subset {\B}(Q',\theta).$$
 
  By Theorem  \ref{al} there exists  an automorphism $h$ such that

\begin{enumerate}[label=(\Roman*)]
\item $\|h-{\rm id}\|_V\leq \delta'_k$,
\item $\|h-\psi\|_U\leq \delta'_k$,
\item  $h(Q_j)=Q_j$ for all $0< j\leq M$,
\item $(h\circ g_{k+1})(Q_M)=Q'$
\item $h(S_j)=S_j$ for all $0< j\leq k+1$,
\item $h(0)=0$, $d_0h={\sf id}$, $h(P_j)=P_j$ for all $1\leq j\leq P_{N_k}$,

\item $h(T_n^{j})=T_n^{j}$ for all $1\leq j\leq k+1$ and $1\leq n\le N_{j-1}+1 $

 \end{enumerate} 
where we have chosen  $\delta_k'\leq \frac{1}{2^k}$ small enough such that
\begin{enumerate}[label=(\roman*)]
\item $(h\circ g_{k+1})^j(\B(P_0,\rho_{k+1}))\subset\subset  \B(P_j,\beta_j)$, for all $1\leq j\leq N_k$
\item $(h\circ g_{k+1})^{N_k+j}(\B(P_0,\rho_{k+1}))\subset\subset  \B(Q_{j}, \theta),$ for all    $0< j\leq  M$
\item $(h\circ g_{k+1})^{N_k+\ell}(\B(P_0,\rho_{k+1}))\subset\subset  \B(Q_{\ell}, s_\ell),$ 
\item $\|(h\circ g_{k+1})(z)-S_{j}\|\leq \frac{k+1}{2(k+2)}\|z-S_{j}\|$ on $\overline{\B}(S_{j},\tau_{j})$ for all $1\leq j\leq k+1$,
\end{enumerate}

\medskip

Similarly as in the proof of Proposition \ref{propgeneral} we define $F_{k+1}:=h\circ g_{k+1}$, so that the sequences of points 
$
(P_j)_{0\leq j\leq N_k}$, $(Q_j)_{0< j\leq M}
$
together form the start of an $F_{k+1}$-orbit, t.i. $F_{k+1}^j(P_0)=P_j$ for $j\leq N_k+M$ where $P_{N_k+j}=Q_j$ for $ 0< j \leq M$.

Set $n_{k+1}:=N_k+\ell$ and $N_{k+1}:=N_k+M$. Define  $\beta_j:=\theta$ for $N_k< j< n_{k+1}$ and for $n_{k+1}<j\leq N_{k+1}$ and $\beta_{n_{k+1}}:=s_\ell$. 

Finally we define $K_{k+1}:=F^{-1}_{k+1}(\B(0,r_{k+1}))$ and observe that  $\B(0,R_{k+1})\subset K_{k+1}$. Since  $$F_{k+1}(\B(P_{N_{k+1}},\beta_{N_{k+1}}))\subset\subset \B(Q',\theta),$$ where  $P_{N_{k+1}}=Q_M$ and since the set
$\B(0,r_{k+1}) \cup  \B(Q',\theta)$ 
is polynomially convex it follows that $K_{k+1}$ and $\overline{\B}(P_{N_{k+1}},\beta_{N_{k+1}})$   are disjoint and their union is polynomially convex.

\medskip

It is immediate that properties $(c)$---$(l)$ are satisfied for the $(k+1)$-th step.  The properties $(a)$ and $(b)$ can be verified by following the last paragraph of the proof of Proposition \ref{propgeneral} verbatim. This concludes the inductive step. $\square$

\medskip

\subsection{Concluding remarks} In Theorem \ref{thmescaping} and Theorem \ref{thmoscillating} the term unit ball can be replaced by any bounded regular open set $\Omega\subset \C^m$ whose closure is polynomially convex.
The simplest example of such set is a bounded convex domain in $\mathbb{C}^n$. Topologically non-trivial examples  can be constructed in the following way. It is well known that any totally real compact manifold $M\subset\mathbb{C}^n$ of dimension $k<n$ can be smoothly perturbed so that its perturbation $M'$ is totally real compact manifold which is polynomially convex, in particular $M'$ has the same topology as $M$. By taking an appropriate tubular neighbourhood of $M'$ we obtain an open set with desired properties (see \cite[Section 3.]{For} and \cite[Theorem 4.13.8]{For}). 
This provides a rich variety of wandering domains which are topologically non-equivalent.  
\medskip 

Here we explain how can one adapt the proof of Proposition \ref{propmain} to include also these domains and note that similarly can be done for the proof of Proposition \ref{propescape}.

 Since each compact polynomially convex set admits a basis of Stein neighbourhoods that are Runge in $\C^m$, there exists a decreasing sequence of compact polynomially convex neighbourhoods $(U_k)$ of $\overline{\Omega}$.

In the above proof we simply replace the role of $\B(P_0,1)$ with $\Omega$ and  $\overline{\B}(P_0,\rho_{k+1})$  with $U_{k+1}$ and choose a point $\tilde{P}_0\in \Omega$.  
Recall that the automorphism $\Phi_1$ defined in \eqref{phi1} and used in \eqref{phi2} maps $\overline{\B}(P_0,\rho_{k+1})$ into $\B(0,1)$ with $\Phi_1(P_0)=0$. We replace this with a automorphism $\tilde{\Phi}_1$ which maps $U_{k+1}$ into $\B(0,1)$ and  satisfies $\tilde{\Phi}_1(\tilde{P}_0)=0$. 

By choosing $\theta$ sufficiently small we may assume that for every $0<j<M$ the ball $\overline{\B}(Q_j,\theta)$ is either contained in $K_k$ or else they are disjoint, hence set $L$ defined in \eqref{defl} is a union of starshapelike domains. Finally in \eqref{A} we get finitely many sets $A_j$, so  that all but one are starshapelike. The only one that might not be starshapelike is the set $\mathcal{W}=F^{N_k}(U_{k+1})$. The rest of the proof follows verbatim.

\subsection*{Acknowledgements}
 I wish thank the referee for thoughtful remarks which helped me to improve the presentation.
Research was supported by the research program P1-0291 from ARRS, Republic of Slovenia

\end{document}